\documentclass[11pt]{article}
\usepackage{color}\usepackage{graphicx}
\usepackage{amsfonts} \usepackage{texdraw} \usepackage{amssymb,
amsmath, amsthm} \usepackage{epsfig} \usepackage[english]{babel}
\usepackage{latexsym} \usepackage{amscd}
\usepackage{bigints}

\parindent7pt \newtheorem{theorem}{Theorem}
 \newtheorem{definition}{Definition}
\newtheorem{lemma}{Lemma}

\def\neweq#1{\begin{equation}\label{#1}} \def\endeq{\end{equation}}
\def\eq#1{(\ref{#1})}

\newcommand{\R}{\mathbb{R}}

\def\sr{\xi(x)}

\begin{document}

\title{Torsional instability in suspension bridges:\\
the Tacoma Narrows Bridge case}

\author{Gianni ARIOLI -- Filippo GAZZOLA }
\date{}
\maketitle

\begin{center}
{\small Dipartimento di Matematica -- Politecnico di Milano \\
        Piazza Leonardo da Vinci 32 - 20133 Milano, Italy\\
{\tt gianni.arioli@polimi.it,\, filippo.gazzola@polimi.it}}
\end{center}

\begin{abstract}
All attempts of aeroelastic explanations for the torsional instability of suspension bridges have been somehow criticised and none of them is unanimously
accepted by the scientific community. We suggest a new nonlinear model for a suspension bridge and we perform numerical experiments with the parameters
corresponding to the collapsed Tacoma Narrows Bridge. We show that the thresholds of instability are in line with those observed the day of the collapse.
Our analysis enables us to give a new explanation for the torsional instability, only based on the nonlinear behavior of the structure.
\end{abstract}

{\bf Keywords:} suspension bridges, torsional instability, Hill equation, modal analysis.

\section{Introduction}

Many suspension bridges manifested aerodynamic instability and uncontrolled oscillations leading to collapses,
see e.g.\ \cite{akesson,bridgefailure}. These accidents are due to several different causes and the focus of this paper is to analyse
those due to wide torsional oscillations. Thanks to the videos available on the web \cite{tacoma}, many people have seen
the spectacular collapse of the Tacoma Narrows Bridge (TNB) occurred in 1940. The torsional oscillations were considered the main cause of the
collapse \cite{ammann,scott}. But the appearance of torsional oscillations is not an isolated event occurred only at the TNB:
among others, we mention the collapse of the Brighton Chain Pier in 1836, the collapse of the Wheeling Suspension Bridge in 1854,
the collapse of the Matukituki Suspension Footbridge in 1977. We refer to \cite{book} for a detailed description of these collapses.\par
These accidents raised some fundamental questions of deep interest for both engineers and mathematicians. Due to the vortex shedding,
longitudinal oscillations are to be expected in suspension bridges, but the reason of the sudden transition from longitudinal to torsional oscillations is less clear.
So far, no unanimously accepted response to this question has been found.
Most attempts of explanations are based on aeroelastic effects such as the frequency of the vortex shedding, parametric resonance, and flutter theory.
The purpose of the present paper is to show that the origin of the torsional instability is purely structural.\par
Von K\'arm\'an, a member of the Board appointed for the Report \cite{ammann}, was convinced that the torsional motion seen on the day of the collapse
was due to the vortex shedding that amplified the already present torsional oscillations and caused the center span to violent twist until the collapse,
see \cite[p.31]{delatte}. But Scanlan \cite[p.841]{scanlan} proved that the frequency of the torsional mode had nothing to do with the natural frequency
of the shed vortices following the von K\'arm\'an vortex pattern: the calculated frequency of a vortex caused by a 68 km/h wind is 1 Hz, whereas the
frequency of the torsional oscillations measured by Farquharson (an engineer witness of the TNB collapse, the man escaping in the video \cite{tacoma})
was 0.2 Hz, see \cite[p.120]{billah}. The conclusion in \cite[p.122]{billah} is that the vortex trail is a consequence, not a primary cause of the
torsional oscillation. Also Green-Unruh \cite[$\S$ III]{green} believe that vortices form independently of the motion and are not responsible
for the catastrophic oscillations of the TNB. The vortex theory was later rediscussed by Larsen \cite[p.247]{larsen}, who stated that vortices may only
cause limited torsional oscillations, but cannot be held responsible for divergent large-amplitude torsional oscillations. Recently, McKenna \cite{mckdcds}
noticed that the behavior described in Larsen's paper was never observed at the Tacoma Bridge and also Green-Unruh \cite{green}
believe that {\em the Larsen model does not adequately explain data or simulations at around 23m/s}.\par
Bleich \cite{bleichsolo} suggested a possible connection between the instability in suspension bridges and the flutter speed of aircraft wings; but
Billah-Scanlan \cite[p.122]{billah} believe that it is a great mistake to relate these two phenomena. Billah-Scanlan also claim that their own work
proves that the failure of the TNB was in fact related to an aerodynamically induced condition of self-excitation in a torsional degree of freedom;
but Larsen \cite[p.244]{larsen} believes that the work in \cite{billah} fails to connect the vortex pattern to the switch of
damping from positive to negative. Moreover, McKenna \cite{mckdcds} states that \cite{billah}
{\em is a perfectly good explanation of something that was never observed, namely small torsional oscillations, and no explanation of what really occurred,
namely large vertical oscillations followed by torsional oscillations}.\par
The parametric resonance method was adapted to a TNB model by Pittel-Yakubovich \cite{pittel0,pittel}, see also \cite[Chapter VI]{yakubovich}
for the English translation and a more general setting. The conclusion on \cite[p.457]{yakubovich} claims that {\em the most dangerous phenomenon
for the stability of suspension bridges is a combination of parametric resonance}. But Scanlan \cite[p.841]{scanlan} comments these attempts by writing
that {\em Others have added to the confusion. A recent mathematics text \cite{yakubovich}, for example, seeking an application for a developed theory
of parametric resonance, attempts to explain the Tacoma Narrows failure through this phenomenon}. We refer to \cite{butikov,herrmann,jenkins} for
connections between (aerodynamic) parametric resonance, flutter theory, and self-oscillations, also applied to suspension bridges.\par
To conclude this quick survey of attempts for aeroelastic explanations, we mention that Scanlan \cite[p.209]{scanlan2} writes that
{\em ...the original Tacoma Narrows Bridge withstood random buffeting for some hours with relatively little harm until some fortuitous
condition ``broke'' the bridge action over into its low antisymmetrical torsion flutter mode}; the words {\em fortuitous condition} tell us
that no satisfactory aeroelastic explanation is available. Due to all these controversial discussions, McKenna \cite[$\S$ 2.3]{mckmonth} writes that
{\em there is no consensus on what caused the sudden change to torsional motion}, whereas Scott \cite{scott} writes that {\em opinion on the exact cause
of the Tacoma Narrows Bridge collapse is even today not unanimously shared}. Summarizing, all the attempts to find a purely aeroelastic explanation of the
TNB collapse fail either because the quantitative parameters do not fit the theoretical explanations or because the experiments in wind tunnels do not
confirm the underlying theory.\par
Nowadays the attention has turned to the nonlinear behavior of structures \cite{arenalacarbonara,lacarbonara}.
In a recent paper \cite{ag13} we gave an explanation in terms of a {\em structural instability}: we
considered an isolated nonlinear bridge model and we were able to show that, if the longitudinal oscillations are sufficiently small, then they
are stable whereas if they are larger they can instantaneously switch to destructive torsional oscillations. The main tool used there are
transfer maps (Poincar\'e maps), which highlight an instability when the characteristic multipliers exit the complex unit circle. The same
phenomenon was later emphasized for different models using the instability for the Hill equation, see \cite{befega,bergaz,bgz}, which can also be
explained by the Floquet theory.
All these results were obtained by considering isolated systems, that is, by neglecting both the aerodynamic forces and the dissipation.
The idea to consider an isolated system was already suggested by Irvine \cite[p.176]{irvine} for a suspension bridge model
similar to the one considered in \cite{ag13}: Irvine ignores damping of both structural and aerodynamic origin, his purpose being to simplify as
much as possible the model by maintaining its essence, that is, the conceptual design of bridges.
And since our purpose is precisely to highlight the role of the structure on the stability, we follow this suggestion: in the conclusions of the
present paper we mention how the aerodynamic effects combine with the structural effects.\par
In this paper we improve the results in \cite{ag13} by introducing a more precise model which also takes into account
the nonlinear restoring action of the cables+hangers system on the deck; the model is described in detail in Section \ref{cabbeam} and is the nonlinear
version of a linear model that we introduced in \cite{agmjm}. In Section \ref{nlm} we define the nonlinear longitudinal modes, that is, the
periodic purely longitudinal motions of the deck which may be unstable and create torsional motions. In Sections \ref{theores} and \ref{nu22} we describe
the theoretical framework that we use to study the stability of the longitudinal modes and we prove that they are stable for small energies.
Finally, we validate our theory by performing numerical simulations of the model with the parameters of the collapsed TNB.
We study the existence and the behavior of approximate longitudinal modes and we compute their instability thresholds.
Our numerical results confirm that the longitudinal oscillations with 8 or 9 nodes
and amplitudes of about 4m, namely the oscillations observed the day of the collapse, are prone to generate torsional oscillations, see Section
\ref{numstab}. The main conclusions and our explanation of the TNB collapse are collected in Section \ref{conclu}.

\section{The mathematical model}\label{cabbeam}

\subsection{Description of the structure}

Throughout this paper the variable $x$ indicates
the position along the deck and we denote the space and time derivatives as follows:
$$w=w(x)\ \Rightarrow\ w'=\frac{dw}{dx}\, ,\qquad\quad
w=w(t)\ \Rightarrow\ \dot{w}=\frac{dw}{dt}\, ,\qquad\quad
w=w(x,t)\ \Rightarrow\ \left\{\begin{array}{ll}
w_x=\frac{\partial w}{\partial x}\\
w_t=\frac{\partial w}{\partial t}
\end{array}\right.$$
and similarly for higher order derivatives. The physical constants are listed in Section \ref{constants}.

In a suspension bridge four towers sustain two cables that, in turn, sustain the hangers. At their lower endpoint the hangers are linked to the
deck and sustain it from above. The hangers are hooked to the cables and the deck is hooked to the hangers.
In this section we quickly revisit the bridge model that we recently introduced in \cite{agmjm}: the displacements involved were assumed to be small,
which justifies the asymptotic expansion of the energies. Here we refine this model by considering more terms: this raises several nonlinearities in the
Euler-Lagrange equations. Moreover, we neglect the flexibility of the hangers: this will be justified below. As usual in engineering
literature, we assume that positive displacements of the deck and the cables are oriented downwards and the origin is at the level of the deck at rest.\par
We assume that the deck has length $L$ and width $2\ell$ with $2\ell\ll L$; we model it as a degenerate plate consisting of a beam representing the
midline and cross sections which can rotate around the beam. The beam contains the barycenters of the cross sections and we denote its position with $y=y(x,t)$.
The angle of rotation of the cross sections with respect to the horizontal position is denoted by $\theta=\theta(x,t)$. Then the positions of the free edges
of the deck are given by $y\pm\ell\sin\theta\approx y\pm\ell\theta$ since we aim to study what happens in a small torsional regime.
We emphasise that our results do not aim to describe the behavior of the bridge when the torsional angle becomes large; instead, we explain how a
small torsional angle can suddenly increase, that is, our purpose is to describe the mechanism that triggers torsional oscillations.\par
Since the spacing between hangers is small relative to the span, the hangers can be considered as a continuous membrane
connecting the cables and the deck. We denote by $-s(x)$ the position of the cables at rest, $-s_0<0$ being the level of the left
and right endpoints of the cables ($s_0$ is the height of the towers); $L$ is the distance between the towers.
If a beam of length $L$ and linear density of mass $M$ is hanged to a cable whose linear density of mass is $m$, and
if $g$ denotes the gravitational constant, then the cable at rest is subject to a downwards vertical force density given by
$\big(M+m\sqrt{1+s'(x)^2}\big)\, g$. Each free side of the deck is connected with hangers to a cable so that
each cable sustains the weight of half deck; then, we obtain the following boundary value problem for $s(x)$:
\neweq{eqparabola}
\begin{cases}
H_0s''(x)=\left(\frac{M}{2}+m\sqrt{1+s'(x)^2}\right)\, g\,,\\
s(0)=s(L)=s_0\,.
\end{cases}
\endeq
It is not difficult to show that \eq{eqparabola} admits a unique solution which, moreover, is symmetric with respect to $x=L/2$, see
\cite[Proposition 1]{agmjm} for the details. In order to simplify notations we set
$$\sr:=\sqrt{1+s'(x)^2}\, ,$$
so that, $\sr$ represents the local length of the cable at rest. When the deck is mounted, the hangers are in tension and reach the length
$s(x)$, where $s$ is the solution of \eq{eqparabola}; if no additional load acts on the system, the equilibrium position of the deck
is horizontal and the position of the cables at equilibrium (deck mounted) is $-s(x)$ for $x\in(0,L)$.\par
The two sustaining cables are labeled with $i=1,2$. If $p_i=p_i(x,t)$ denotes the displacement of the cable $i$ with respect to the rest position,
then $p_i(x,t)-s(x)$ denotes the actual position of the cable.\par
The flexibility of the hangers has a significant effect on the frequencies of the higher modes when the deck is very stiff, see \cite{luco}.
However, as far as lower modes and weakly stiffened bridges are involved, the elastic deformation of the hangers may be neglected, see again
the results by Luco-Turmo \cite{luco}. Since we aim to model a very flexible deck as the TNB, we assume that the hangers are rigid, their action merely
being to connect the deck to the cables. This assumption is justified by the fact that the action of the cables is the main cause of the nonlinearity
of the restoring force, see Bartoli-Spinelli \cite[p.180]{bartoli}.
Moreover, the nonlinear contribution of the hangers is mainly due to their slackening but, as reported in \cite{ammann}, the slackening
of the hangers occurred at the TNB only {\em after} that the large torsional oscillations appeared. Our purpose is to understand
how negligible torsional oscillations of the deck suddenly become dangerous ones, that is, to describe what happens {\em before}
the slackening starts. This is a further reason to consider rigid hangers.
If we assume that the hangers have a fixed length, then $p_1=y+\ell\theta$ and $p_2=y-\ell\theta$: for our convenience, in the next subsection
we maintain the notation with the $p_i$'s while we replace them in the Euler-Lagrange equations \eq{caps1}-\eq{caps2}.

\subsection{Kinetic and potential energies of the structure}

The kinetic energy of the deck is the sum of the kinetic energy of the barycenter of the cross section and of the kinetic energy of the torsional angle.
Then the total kinetic energy of the bridge is
$$
E_{\rm{kin}}=\frac{M}{2}\int_0^L\left(\frac{\ell^2\theta_t^2}{3}+y_t^2\right)\, dx+\frac{m}{2}\int_0^L[(p_1)_t^2+(p_2)_t^2]\sr\, dx\, .
$$

The gravitational energies of the deck and the cables are, respectively, given by
$$
-Mg\int_0^L y\,dx\qquad\mbox{and}\qquad-mg\int_0^L (p_1+p_2)\sr\,dx\,.
$$

The elastic energy of the deck is composed by the bending energy of the beam and the torsional energy:
$$
E_{\rm{el}}=\frac{EI}{2}\int_0^L y_{xx}^2\, dx+\frac{GK}{2}\int_0^L \theta_x^2\, dx\, .
$$

The tension of each cable consists of two parts. The first part is the tension at rest $H(x)=H_0\sr$.
The amount of energy needed to deform each cable at rest under the tension $H(x)$ in the infinitesimal interval $[x,x+dx]$
from the original position $-s(x)$ to $-s(x)+p_i(x,t)$ ($i=1,2$) is the variation of length times the tension, that is
$$
E_c(x)\,dx=H_0\sr\, \Big(\sqrt{1+[s'(x)-(p_i)_x(x,t)]^2}-\sqrt{1+s'(x)^2}\Big)\, dx\,.
$$
Hence, using the asymptotic expansion
\neweq{asymptotic}
\sqrt{1+[s'-(p_i)_x]^2}-\sqrt{1+(s')^2}=-\frac{s'(p_i)_x}{\xi}+\frac{(p_i)_x^2}{2\xi^3}+\frac{s'(p_i)_x^3}{2\xi^5}+o\Big((p_i)_x^3\Big)
\endeq
and neglecting the higher order terms, the energy necessary to deform the whole cable is
\neweq{defenergy}
E_c(p_i)=\int_0^L E_c(x)\,dx\approx H_0\int_0^L\left(-s'(p_i)_x+\frac{(p_i)_x^2}{2\xi^2}+\frac{s'(p_i)_x^3}{2\xi^4}\right)\,dx\,.
\endeq
The second part of the tension of each cable is the additional tension
$$
\frac{AE}{L_c}\, \Gamma(p_i)
$$
due to the increment of length $\Gamma(p_i)$ of the cable ($i=1,2$). This increment requires the energy
\neweq{Etc}
E_{tc}(p_i)=\frac{AE}{2L_c}\,\Gamma(p_i)^2\,.
\endeq
In view of the asymptotic expansion \eq{asymptotic}, we can approximate the increment $\Gamma(p_i)$ of the length of
the cable (due to its displacement $p_i$) by
$$\Gamma(p_i):=\int_0^L\sqrt{1+[s'-(p_i)_x]^2}\,dx-L_c\approx
\int_0^L\left(\frac{-s'(p_i)_x}{\xi}+\frac{(p_i)_x^2}{2\xi^3}\right)\, dx\, .$$
A third order expansion then gives
$$\Gamma(p_i)^2=\left(\int_0^L \frac{s'(p_i)_x}{\xi}\right)^2-
\left(\int_0^L \frac{s'(p_i)_x}{\xi}\right)\left(\int_0^L\frac{(p_i)_x^2}{\xi^3}\right)+o\Big((p_i)_x^3\Big)\, .$$
Therefore, we approximate \eq{Etc} with
\neweq{Etc2}
E_{tc}(p_i)=\frac{AE}{2L_c}\left(\int_0^L \frac{s'(p_i)_x}{\xi}\right)^2
-\frac{AE}{2L_c}\left(\int_0^L \frac{s'(p_i)_x}{\xi}\right)\left(\int_0^L\frac{(p_i)_x^2}{\xi^3}\right)\, .
\endeq

\subsection{The Euler-Lagrange equations}

Some integration by parts show that the variation of the energy $E_c(p_i)$ in \eq{defenergy} is given by
$$
dE_c(p_i)z= H_0\int_0^L\left(s'-\frac{(p_i)_x}{\xi^2}-\frac{3s'(p_i)_x^2}{2\xi^4}\right)_x z\, dx\, .
$$
For the variation of the energy $E_{tc}(p_i)$ in \eq{Etc2}, we neglect terms of order larger than 2, and by integrating by parts we obtain
$$
dE_{tc} (p_i)z =
\frac{AE}{2L_c}\left(\int_0^L\frac{(p_i)_x^2}{\xi^3}\right)\int_0^L \frac{s''}{\xi^3}\, z\, dx
+\frac{AE}{L_c}\left(\int_0^L \frac{s''p_i}{\xi^3}\right)\int_0^L\left(\frac{s'}{\xi}-\frac{(p_i)_x}{\xi^3}\right)_x\, z\, dx\, .
$$

By recalling that $p_1=y+\ell\theta$ and $p_2=y-\ell\theta$, the first variation of the sum of all the energies involved yields the following
Euler-Lagrange equations:
\begin{eqnarray}
\!\!\!\!\!\!\big(M\!+\!2m\xi\big)y_{tt} &\!\!\!\!=\!\!\!\!& -EIy_{xxxx}+
H_0\left(\frac{2y_x}{\xi^2}+3\frac{s'(y_x^2+\ell^2\theta_x^2)}{\xi^4}\right)_x
-\frac{AE}{L_c}\left[\int_0^L\frac{y_x^2+\ell^2\theta_x^2}{\xi^3}\right]\!\frac{s''}{\xi^3} \notag\\
 & &-\frac{2AE}{L_c}\left[\int_0^L \frac{s''y}{\xi^3}\right]\!\left(\frac{s'}{\xi}\!-\!\frac{y_x}{\xi^3}\right)_x
+\frac{2AE\ell^2}{L_c}\left[\int_0^L \frac{s''\theta}{\xi^3}\right]\!\left(\frac{\theta_x}{\xi^3}\right)_x, \label{caps1}\\
\!\!\!\!\!\!\left(\tfrac{M}{3}\!+\!2m\xi\right)\theta_{tt} &\!\!\!\!=\!\!\!\!&
\frac{GK}{\ell^2}\theta_{xx}+2H_0\left(\frac{\theta_x}{\xi^2}+3\frac{s'y_x\theta_x}{\xi^4}\right)_x
-\frac{2AE}{L_c}\left[\int_0^L\frac{y_x\theta_x}{\xi^3}\right]\!\frac{s''}{\xi^3} \notag \\
 & &-\frac{2AE}{L_c}\left[\int_0^L \frac{s''\theta}{\xi^3}\right]\!\left(\frac{s'}{\xi}\!-\!\frac{y_x}{\xi^3}\right)_x
+\frac{2AE}{L_c}\left[\int_0^L \frac{s''y}{\xi^3}\right]\!\left(\frac{\theta_x}{\xi^3}\right)_x. \label{caps2}
\end{eqnarray}
This is the {\bf system describing the dynamics of a suspension bridge} that we will consider throughout the paper.
Since the degenerate plate is hinged between the two towers and the cross sections between the towers cannot rotate,
the boundary conditions to be associated to \eq{caps1}-\eq{caps2} read
\neweq{bc}
y(0,t)=y(L,t)=y_{xx}(0,t)=y_{xx}(L,t)=\theta(0,t)=\theta(L,t)=0\quad\forall t\ge0\, .
\endeq

For our investigation of \eq{caps1}-\eq{caps2} (both theoretical and numerical) we first set $\theta\equiv0$ in \eq{caps1} to obtain
\neweq{nony}
(M\!+\!2m\xi)y_{tt}\!+\!EI y_{xxxx}\!=\!
H_0\left(\frac{2y_x}{\xi^2}\!+\!\frac{3s'y_x^2}{\xi^4}\right)_x\!-\!
\frac{AE}{L_c}\left[\int_0^L\!\frac{y_x^2}{\xi^3}\right]\frac{s''}{\xi^3}\!
-\!\frac{2AE}{L_c}\left[\int_0^L\!\frac{s''y}{\xi^3}\right]\left(\frac{s'}{\xi}\!-\!\frac{y_x}{\xi^3}\right)_x
\endeq
and we find periodic solutions $y$ of different amplitude and number of zeros. Then, for each solution $y$ of \eq{nony} we consider \eq{caps2}
as a linear equation in $\theta$ with periodic coefficients and we study the stability of the trivial solution of \eq{caps2}.\par
Our characterization of torsional stability consists in taking a periodic solution $y$ of \eq{nony} and plugging it into \eq{caps2}; we say that
$y$ is {\em torsionally stable} if the trivial solution of the so obtained equation (see \eq{theta1} below) is stable. This is basically the characterization
suggested by Dickey \cite{dickey3} for a couple of oscillators related to a nonlinear string: it is also called {\em linear stability} in literature
\cite[Definition 2.3]{ghg}. An apparently stronger form of stability, without linearization and substitution of $y$, is the well-known Lyapunov
characterization. For a simpler system of ODEs Ghisi-Gobbino \cite{ghg} were able to prove that these two characterizations are equivalent.

\section{Nonlinear longitudinal modes}\label{nlm}

The milestone paper by Rabinowitz \cite{rabin} started the systematic study of the existence of periodic solutions for the nonlinear string equation
\neweq{nonlinw}
u_{tt}-u_{xx}+f(x,u)=0\quad(x,t)\in(0,\pi)\times\R_+\, ,\qquad u(0,t)=u(\pi,t)=0\quad t\in\R_+\, .
\endeq
Prior to \cite{rabin}, only perturbation techniques were available and the nonlinearity was assumed to be small in a suitable sense.
Rabinowitz proved that \eq{nonlinw} admits periodic solutions. In fact, there exist infinitely many periodic solutions of \eq{nonlinw}, having
``almost'' any possible period $T>0$: quite sophisticated tools are necessary to obtain a full description of the set of periodic solutions,
see e.g.\ \cite{gentile} and references therein. Periodic solutions were also found for the related nonlinear beam equation
\neweq{nonlinb}
\left\{\begin{array}{ll}
u_{tt}+u_{xxxx}+f(x,u)=0\quad & (x,t)\in(0,\pi)\times\R_+\, ,\\
u(0,t)=u(\pi,t)=u_{xx}(0,t)=u_{xx}(\pi,t)=0\quad & t\in\R_+\, ,
\end{array}\right.
\endeq
see \cite{lee,liu1,liu2}.
Also nonlocal (and nonlinear) variants of \eq{nonlinw} have been considered. In 1876 Kirchhoff \cite{kirch} introduced the following nonlinear vibrating
string equation
\neweq{nonlocw}
u_{tt}-\Big(a+b\int_0^\pi\!\!u_x^2\Big)u_{xx}=0\quad(x,t)\in(0,\pi)\times\R_+\, ,\qquad u(0,t)=u(\pi,t)=0\quad t\in\R_+\ ,
\endeq
where $a,b>0$. Note that \eq{nonlinw} with $f(x,u)=au+bu^3$ is quite similar to \eq{nonlocw} and, in some respect, \eq{nonlocw} appears more
complicated since first order derivatives are involved. Nevertheless, the modal analysis of \eq{nonlocw} is much simpler since it behaves as a ``hidden''
linear problem. Exploiting this peculiarity, Cazenave-Weissler \cite{cazw2} proved that \eq{nonlocw} has infinitely many periodic solutions.\par
If we take $\theta(x,0)=\theta_t(x,0)=0$ as initial data in \eq{caps1}-\eq{caps2}-\eq{bc}, then $\theta(x,t)\equiv0$, whereas $y$ solves the
nonlinear nonlocal equation \eq{nony}. A careful look at \eq{nony} shows that its structure, its nonlinearities, and its nonlocal terms are of the
same kind as those appearing in \eq{nonlinw}, \eq{nonlinb}, and \eq{nonlocw}.
It is therefore reasonable to conjecture that \eq{nony} admits infinitely many periodic-in-time solutions.\par
In order to find such solutions, or at least a reliable approximation of them, we proceed numerically.
The first step is to look for small periodic solutions of \eq{nony} close to those of the linearized problem
\neweq{linearized}
(M+2m\xi)y_{tt}+EIy_{xxxx}=2H_0\left(\frac{y_x}{\xi^2}\right)_x
-\frac{2AE}{L_c}\left[\int_0^L\frac{s''y}{\xi^3}\right]\frac{s''}{\xi^3}\, .
\endeq
Let us analyze \eq{linearized} and see what has to be expected from the numerical procedure.
Periodic solutions of \eq{linearized} may be obtained by separating variables.
Consider the linear operator ${\mathcal L}$ defined by
$${\mathcal L}u:=EI\, u''''-2H_0\left(\frac{u'}{\xi^2}\right)'+
\frac{2AE}{L_c}\left[\int_0^L\frac{s''u}{\xi^3}\right]\frac{s''}{\xi^3}\qquad\forall u\in C^4[0,L]\, .
$$
The definition of ${\mathcal L}$ may be continuously extended to the Sobolev space $H^2\cap H^1_0(0,L)$: the operator ${\mathcal L}$ maps this
space into its dual space $H'$ and is implicitly defined by
\neweq{implicit}
\langle{\mathcal L}u,v\rangle=
EI\int_0^L u''v''+2H_0\int_0^L\frac{u'v'}{\xi^2}+\frac{2AE}{L_c}\left[\int_0^L\frac{s''u}{\xi^3}\right]\left[\int_0^L\frac{s''v}{\xi^3}\right]\qquad
\forall u,v\in H^2\cap H^1_0(0,L)
\endeq
where $\langle\cdot,\cdot\rangle$ denotes the duality pairing between $H^2\cap H^1_0(0,L)$ and $H'$. The bilinear form
in \eq{implicit} is symmetric and coercive. Therefore, the operator ${\mathcal L}$ is self-adjoint and the eigenvalue problem
\neweq{eigen}
{\mathcal L}u=\lambda(M+2m\xi)u
\endeq
admits an unbounded sequence $\{\lambda_k\}$ of positive eigenvalues whose corresponding eigenfunctions form a complete system in $H^2\cap H^1_0(0,L)$.\par
If the cables were in the limit horizontal position, that is $s'=s''=0$ and $\xi=1$, then the eigenfunctions of \eq{eigen} would be $\sin(\frac{k\pi}{L}x)$
for all integer $k\ge1$ with corresponding eigenvalues given by
$$\lambda_k=\frac{EI\, k^4\pi^4+2H_0\, k^2\pi^2L^2}{(M+2m)L^4}\, .$$
As we could verify numerically, when the shape of the cable satisfies \eq{eqparabola}, the eigenfunctions of \eq{eigen} remain
close to $\sin(\frac{k\pi}{L}x)$ with $k\!-\!1$ zeros in $(0,L)$.\par
Denote by $\varphi_\lambda$ an eigenfunction of \eq{eigen} relative to the eigenvalue $\lambda$, then for any $A,B\in\R$ the function
\neweq{u}
u(x,t)=\varphi_\lambda(x)\left(A\sin\big(\sqrt{\lambda}\, t\big)+B\cos\big(\sqrt{\lambda}\, t\big)\right)
\endeq
is a periodic solution of \eq{linearized}.

We numerically find approximate periodic solutions for \eq{nony} in two steps. First, we transform the PDE into a system of
ODEs by discretising in space via Fourier polynomials and we build a Mathematica program that solves the corresponding initial value problems.
Then we look for points in the (finite dimensional) phase space that belong to a periodic orbit.\par
More precisely, given the boundary conditions \eq{bc}, we look for an approximate solution
of \eq{nony} that can be represented by a trigonometric polynomial
\neweq{trigpol}
y(x,t)=\sum_{k=1}^n y_k(t)\sin\left(\frac{k \pi x}{L}\right)\,.
\endeq
To achieve this goal, we plug \eq{trigpol} into \eq{nony} and we project it onto the span of
$\{\sin\left(\pi x/L\right),\ldots,\sin\left(n\pi x/L\right)\}$: in particular, we choose initial conditions
$$y(x,0)=\sum_{k=1}^n y_{0k}\sin\left(\frac{k \pi x}{L}\right)\,,\quad y_t(x,0)=\sum_{k=1}^n y_{1k}\sin\left(\frac{k \pi x}{L}\right)$$
for some $Y_0=\{y_{0k}\}_{k=1,\ldots,n}\in\R^n$ and $Y_1=\{y_{1k}\}_{k=1,\ldots,n}\in\R^n$.
The dimension $n$ is chosen in such a way that the result of the numerical integration does not change significantly when $n$ is increased.
For the computation of the first $6$ periodic solutions we take $n=10$, while for the computation of the next $4$ periodic solutions we take $n=16$.

Let $Y(t)=\{y_{k}(t)\}_{k=1,\ldots,n}\in\R^n$; we obtain a system of $n$ nonlinear ODEs in the form
\neweq{systode}
\ddot{Y}(t)=G\big(Y(t)\big)\qquad(G:\R^n\to\R^n)
\endeq
with initial conditions $Y(0)=Y_0$ and $\dot{Y}(0)=Y_1$. We use Newton's algorithm to find initial vectors $Y_0$ and $Y_1$ that lead to periodic
solutions of \eq{systode}. For any $T>0$ we define the transfer map
$$\Phi_T:\R^{2n}\to\R^{2n}\ ,\qquad\Phi_T(Y_0,Y_1)=\big(Y(T),\dot Y(T)\big)\, ,$$
where $Y(t)$ is the solution of \eq{systode} with initial conditions $(Y_0,Y_1)$, and we look for fixed points of $\Phi_T$ by iterating the transformation
$$
N(Y_0,Y_1):=(Y_0,Y_1)-[J\Phi_T(Y_0,Y_1)-I]^{-1}(\Phi_T(Y_0,Y_1)-(Y_0,Y_1))
$$
where $J\Phi_T$ is the Jacobian matrix of $\Phi_T$ and $I$ is the identity $2n\times2n$ matrix.
As initial point for the Newton's algorithm we choose $Y_0=\alpha e_k$ and $Y_1=0$, where $\alpha$ is some positive number,
$k\in\{1,\ldots,n\}$ and $\{e_1,...,e_n\}$ is the canonical basis in $\mathbb{R}^n$. We observe that, if $\alpha$ is sufficiently small,
there exists $T>0$ such that $\Phi_T(Y_0,Y_1)\simeq(Y_0,Y_1)$: for such $T$, Newton's algorithm converges rapidly.
In Table \ref{almostlinear} we quote the period $T$ of each mode obtained for $\alpha\to0$.
\begin{table}[htdp]
\begin{center}
{\small
\begin{tabular}{|c|c|c|c|c|c|c|c|c|c|c|}
\hline
Mode & 1 & 2 & 3 & 4 & 5 & 6 & 7 & 8 & 9 & 10 \\
\hline
Period ($s$) & 10.95 & 7.67 & 5.42 & 3.75 & 2.9 & 2.41 & 2.02 & 1.72 & 1.5 & 1.32 \\
\hline
\end{tabular}
}
\end{center}
\caption{Period of the longitudinal mode on each branch for energies close to 0.}\label{almostlinear}
\end{table}

Figure \ref{graph8} represents the graph of the solution obtained by this algorithm for $k=8$ at time $t=0$;
as expected, it resembles a multiple of $\sin(\frac{8\pi}{L}x)$ (with $L=853.44$m, see Section \ref{constants}).
Since this function has small $L^\infty$-norm, we expect it to be an approximate periodic solution of the linear equation \eq{linearized}.
\begin{figure}[ht]
\begin{center}
{\includegraphics[height=40mm, width=100mm]{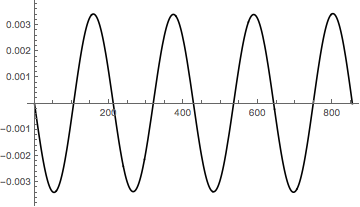}}
\caption{$y(x,0)$ for the solution on the 8th branch at $T=1.72$ (small energy).}\label{graph8}
\end{center}
\end{figure}

This small periodic solution of the nonlinear system \eq{systode} is an example of what we expect to observe when a still bridge is slightly perturbed by some external forcing, typically the wind. If more energy is inserted
into the structure, the oscillation becomes wider. Since the insertion of energy is gradual, it seems natural to look for other solutions, of larger energy and $L^\infty$-norm, by
continuity. To do so, we slightly increase $T$ and use again Newton's algorithm: it turns out that, on each branch, the period is an increasing function
of the energy and, if the increment in $T$ is sufficiently small, Newton's algorithm keeps converging rapidly, and the resulting solution of \eq{systode}
has larger energy and $L^\infty$-norm. By increasing $T$ we build the branch of the $k$-th nonlinear modes of \eq{systode}.
This procedure cannot be extended indefinitely; at some point the convergence of Newton's algorithm starts deteriorating and smaller increments in $T$
must be taken in order to keep following the branch of solutions. We stop the procedure when the acceptable increments become too small to be feasible.
In this way we obtain all the branches of the $k$-th nonlinear longitudinal modes for $k=1,\ldots,10$.

\begin{definition}\label{nnllmm}
We call a periodic solution of \eqref{systode} obtained with this algorithm a {\bf $k$-th nonlinear longitudinal mode of period $T$}.
\end{definition}

We observed that the nonlinear longitudinal modes on the odd branches are symmetric with respect to the center of the bridge, see e.g.\ Figure \ref{7}
where the solution on the 7th branch at $T=2.18$ is displayed.
On the other hand, the modes on even branches have no symmetries, see e.g.\ Figure \ref{8} where the solution on the 8th branch at $T=1.86$ is displayed.
\begin{figure}[ht]
\begin{center}
{\includegraphics[height=45mm, width=100mm]{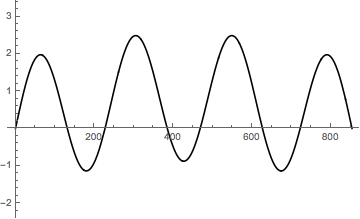}}
\caption{$y(x,0)$ for the solution on the 7th branch at $T=2.18$ (critical threshold of instability).}\label{7}

\vskip7mm
{\includegraphics[height=40mm, width=100mm]{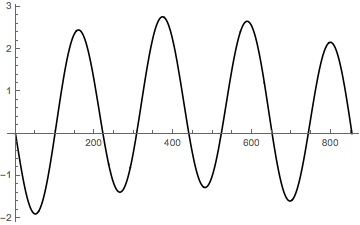}}
\caption{$y(x,0)$ for the solution on the 8th branch at $T=1.86$ (critical threshold of instability).}\label{8}
\end{center}
\end{figure}

We remark that the $k$-th Fourier component of a $k$-th nonlinear longitudinal mode is much larger than the other ones, see e.g.\ Figure \ref{8thmode}
where the evolution in time of the first 12 Fourier components for the solution obtained with $k=8$ with $T=1.86$ is displayed.

In all our experiments we also observed that the first Fourier component of the periodic solution (namely the function $y_1(t)$ in \eq{trigpol})
is negative for all times, see e.g.\ the first plot in Figure \ref{8thmode}. This means that the deck is bent upwards during periodic motions.
In fact, our results show that, when the periodic motion starts, the cables undergo an additional tension due to the nonlocal term and therefore
they pull the deck upwards.
\begin{figure}[htdp]
\begin{center}
{\includegraphics[height=35mm, width=75mm]{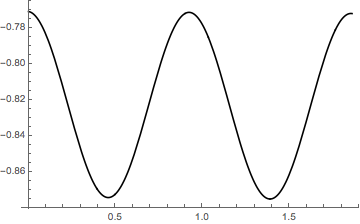}}
{\includegraphics[height=35mm, width=75mm]{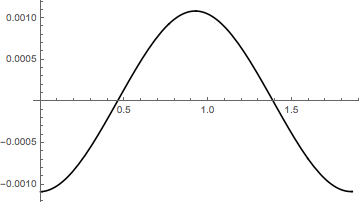}}\\
{\includegraphics[height=35mm, width=75mm]{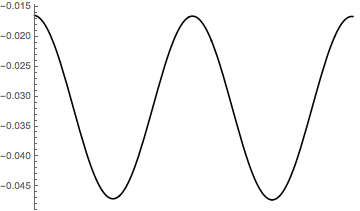}}
{\includegraphics[height=35mm, width=75mm]{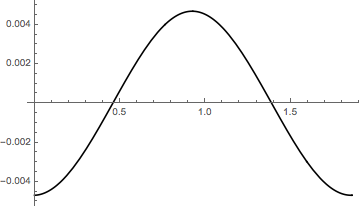}}\\
{\includegraphics[height=35mm, width=75mm]{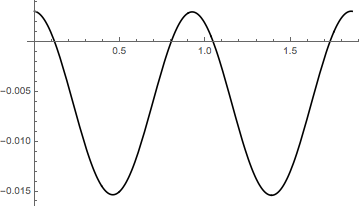}}
{\includegraphics[height=35mm, width=75mm]{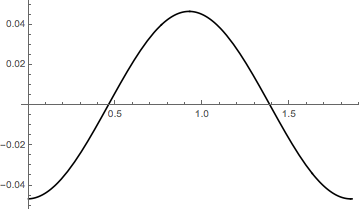}}\\
{\includegraphics[height=35mm, width=75mm]{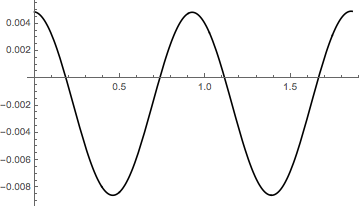}}
{\includegraphics[height=35mm, width=75mm]{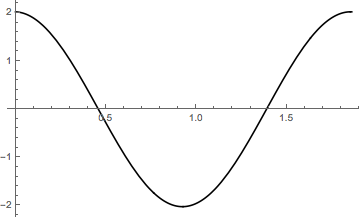}}\\
{\includegraphics[height=35mm, width=75mm]{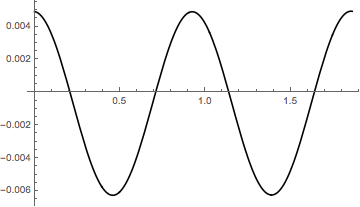}}
{\includegraphics[height=35mm, width=75mm]{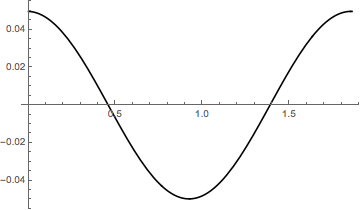}}\\
{\includegraphics[height=35mm, width=75mm]{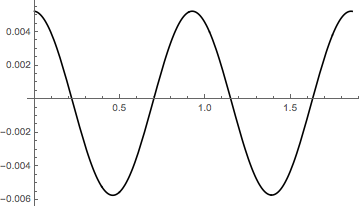}}
{\includegraphics[height=35mm, width=75mm]{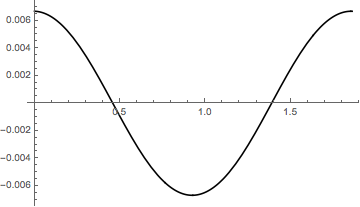}}
\caption{Time evolution of the Fourier components 1 to 12 of the solution on the 8th branch for $T=1.86$ (critical threshold of instability).}\label{8thmode}
\end{center}
\end{figure}
\vfill\eject

\section{Torsional stability of longitudinal modes}

\subsection{Theoretical framework}\label{theores}

The purpose of this section is to insert the stability analysis of suspension bridges within a suitable theoretical framework.
We explain why,
for small longitudinal oscillations, one expects the system \eq{caps1}-\eq{caps2} to remain torsionally stable, and why, for large longitudinal
oscillations, the system becomes unstable. Most of the tools introduced in this section are applications to our context of classical tools
from the theory of ODE's that can be found e.g.\ in \cite{chicone,yakubovich}.\par
We consider a (time) periodic solution $y$ of \eq{nony} and we denote by $T_y$ its period.
Then we view the equation \eq{caps2} for $\theta$ as an independent equation, that is, a linear hyperbolic equation of the form
\neweq{theta1}
\alpha(x)\theta_{tt}=\Big(a(x,t)\theta_x\Big)_x+\beta(x)\left[\int_0^Lb(x,t)\theta(x,t)dx\right]+c(x,t)\left[\int_0^L\gamma(x)\theta(x,t)dx\right]
\endeq
for suitable coefficients $a,b,c,\alpha,\beta,\gamma$; note that the coefficients $a,b,c$ depend on the periodic solution $y=y(x,t)$ of \eq{nony}.

Because of energy conservation, if the energy of the system is at some time concentrated on low Fourier coefficients of the solution, e.g.\ if the initial
condition is smooth, then it remains so for all times. Hence, it seems sensible to simplify the problem by studying its projection onto the
space spanned by
$\{\sin(\frac{\pi}{L}x),...,\sin(\frac{\nu\pi}{L}x)\}$, for some integer $\nu\ge2$. This integer $\nu$ needs not
(so che a volte si usa, ma non \`e corretto. "does not need" \`e meglio)
to be the same as the integer $n$
in Section \ref{nlm}: in order to state and prove a sufficient condition for stability we simply take $\nu=2$ whereas for our numerical experiments
we take $\nu=n$.\par
For each $k=1,...,\nu$ we multiply \eq{theta1} by $\sin(\frac{k\pi}{L}x)$
and we integrate over $(0,L)$: we obtain $\nu$ ODEs in the form
\neweq{theta3}
\ddot{\theta}_k(t)+\sum_{j=1}^\nu \chi_{jk}(t)\theta_j(t)=0\,,\qquad(k=1,...,\nu)
\endeq
for some coefficients $\chi_{ij}$ depending only on $t$. We take the solution of \eq{theta3}, that is,
\neweq{O}
\theta(x,t)=\sum_{j=1}^\nu \theta_j(t)\sin\left(\frac{j\pi}{L}x\right)
\endeq
as an approximate solution of  \eq{theta1}. Since the time-dependent coefficients in \eq{theta1} are computed in terms of the periodic solution $y$
of \eq{nony}, we have that
\neweq{trueperiod}
\mbox{the functions}\quad t\mapsto\chi_{jk}(t)\quad\mbox{are }T_y\mbox{\,-\,periodic }\forall j,k=1,...,\nu\, .
\endeq
Moreover, these coefficients are symmetric: $\chi_{ij}(t)\equiv\chi_{ji}(t)$ for all $i,j$. We define the symmetric $\nu\times\nu$ matrix
$$
{\bf \Xi}(t)\ =\ \Big(\chi_{jk}(t)\Big)_{j,k=1,...,\nu}\, ,
$$
which is $T_y$\,-\,periodic by \eq{trueperiod}. Then we rewrite \eq{theta3} in the vectorial form
\neweq{theta5}
\ddot{W}+{\bf \Xi}(t)W=0\,,
\endeq
where $W(t)=(\theta_1(t),...,\theta_\nu(t))$ is the finite dimensional approximate vector containing the torsional components of the motion.
Following \cite[p.155, vol.1]{yakubovich} we give the

\begin{definition}
We say that the system \eqref{theta5} is {\bf strongly stable} if all its solutions are bounded in $\R$ and
this property is preserved under small perturbations of ${\bf \Xi}(t)$.
\end{definition}

In turn, since the system \eq{theta5} has coefficients which depend on the particular nonlinear longitudinal mode used to build the
symmetric matrix $\Xi(t)$, we may characterize the stability of the modes as follows.

\begin{definition}
We say that a nonlinear longitudinal mode of \eqref{systode} (see Definition \ref{nnllmm}) is {\bf torsionally stable} if the system \eqref{theta5}
is strongly stable; otherwise, we say that it is {\bf torsionally unstable}.
\end{definition}

Clearly, the boundedness of all solutions of the linear system \eq{theta5} is equivalent to the stability of the trivial solution $W(t)\equiv0$ and,
therefore, to the torsional stability of the nonlinear longitudinal mode.
Our purpose is precisely to study the stability of \eq{theta5}. In Section \ref{nu22} we state and prove a sufficient condition for
the stability in the case $\nu=2$ and we discuss how, for general systems such as \eq{theta5}, instability may arise.
In Section \ref{numstab} we compute numerically the thresholds of instability in dependence of the behavior of the nonlinear longitudinal mode
which is used to built the coefficient matrix $\Xi$ in \eq{theta5}.

\subsection{The case $\nu=2$}\label{nu22}

A fairly useful tool to prove the stability for linear systems such as \eq{theta5} is the following criterion,
see \cite[Theorem II p.270, vol.1]{yakubovich}:

\begin{lemma}\label{suff}
Assume that a $\nu\times \nu$ matrix ${\bf P}(t)$ is continuous and $T$-periodic for some $T>0$ and write ${\bf P}(t)={\bf D}(t)+{\bf R}(t)$ where
${\bf D}(t)$ is the diagonal matrix containing the same elements as ${\bf P}(t)$ and ${\bf R}(t)={\bf P}(t)-{\bf D}(t)$ is the ``residual matrix''
containing all the coefficients off the diagonal of ${\bf P}(t)$. Let $r_1$ and $r_2$ be two continuous $T$-periodic functions such that
$$r_1(t){\bf I}\le{\bf R}(t)\le r_2(t){\bf I}\qquad\forall t\in[0,T]\, .$$
Let $\delta_k(t)$ ($k=1,...,\nu$) be the elements of the diagonal matrix ${\bf D}(t)$. Assume that for all $s\in[0,1]$ the scalar Hill equations
\neweq{tantiscalari}
\ddot{z}_k+\Big(\delta_k(t)+sr_1(t)+(1-s)r_2(t)\Big)z_k(t)=0\qquad(k=1,...,\nu)
\endeq
are strongly stable. Then the system $\ddot{Z}+{\bf P}(t)Z=0$ is strongly stable (here, $Z=Z(t)\in\R^\nu$).
\end{lemma}

In this section we study in detail the case $\nu=2$, since this case is simpler to handle and allows to prove a sufficient condition for the stability.
Moreover, according to the detailed analysis of the TNB collapse by Smith-Vincent \cite[p.21]{tac2},
{\em the only torsional mode which developed under wind action on the bridge or on the model is that with a single node at the center of the main span},
see also \cite{book} for further evidence showing the importance of the second torsional component.\par
If $\nu=2$ then \eq{theta5} simply becomes
\neweq{theta6}
\left(\begin{array}{c}
\!\ddot{\theta}_1\! \\
\!\ddot{\theta}_2\!
\end{array}\right)+
\left(\begin{array}{cc}
\chi_{11}(t) & 0 \\
0 & \chi_{22}(t)
\end{array}\right)\!
\left(\begin{array}{c}
\!\theta_1\! \\
\!\theta_2\!
\end{array}\right)+
\left(\begin{array}{cc}
0 & \chi_{12}(t) \\
\chi_{12}(t) & 0
\end{array}\right)\!
\left(\begin{array}{c}
\!\theta_1\! \\
\!\theta_2\!
\end{array}\right)=\left(\begin{array}{c}
\!0\! \\
\!0\!
\end{array}\right)
\endeq
where we have emphasized the diagonal matrix ${\bf D}(t)$ and the residual matrix ${\bf R}(t)$. We point out that a simple and elegant form such as
\eq{theta6} is made possible by the fact that $\nu=2$: the eigenvalues of the residual (symmetric) matrix ${\bf R}(t)$ are $\pm\chi_{12}(t)$ which means that
\neweq{eigenroots}
-|\chi_{12}(t)|{\bf I}\le{\bf R}(t)\le|\chi_{12}(t)|{\bf I}\qquad\forall t\in[0,T_y]\, .
\endeq

By \eq{eigenroots}, in the situation under study the $\nu=2$ families of Hill equations in \eq{tantiscalari} read
\neweq{twofamilies}
\begin{array}{cc}
\ddot{\theta}_1+\Big(\chi_{11}(t)+\alpha\chi_{12}(t)\Big)\theta_1=0\\
\ddot{\theta}_2+\Big(\chi_{22}(t)+\alpha\chi_{12}(t)\Big)\theta_2=0
\end{array}\qquad\qquad(-1\le\alpha\le1)\, .
\endeq

Let us derive \eq{twofamilies} in the limit case where $y=0$.

\begin{lemma}\label{limit}
There exist $\gamma_1,\gamma_2>0$ such that, if $y\equiv0$, then the function
\neweq{nu2}
\theta(x,t)=\theta_1(t)\sin\left(\frac{\pi}{L}x\right)+\theta_2(t)\sin\left(\frac{2\pi}{L}x\right)
\endeq
satisfies the projection of \eqref{caps2} onto $\mbox{\rm span}\{\sin(\frac{\pi}{L}x),\sin(\frac{2\pi}{L}x)\}$ with $\theta_1$ and $\theta_2$ solving
\neweq{limits}
\ddot{\theta}_1+\gamma_1\theta_1=0\, ,\quad\ddot{\theta}_2+\gamma_2\theta_2=0\, .
\endeq
\end{lemma}
\begin{proof} If $y\equiv0$, then \eq{caps2} reads
\neweq{zerozero}
\left(\frac{M}{3}+2m\xi\right)\ell^2\theta_{tt}=GK\theta_{xx}+2\ell^2\bigg\{H_0\left(\frac{\theta_x}{\xi^2}\right)_x
-\frac{AE}{L_c}\left[\int_0^L\frac{s''\theta}{\xi^3}\right]\frac{s''}{\xi^3}\bigg\}
\endeq
and all the coefficients are now independent of $t$. The functions involved in this equation are either symmetric or skew-symmetric with respect
to $x=\frac{L}{2}$:
\neweq{skew}
s''(\!x\!),\xi(\!x\!),\sin(\!\tfrac{\pi}{L}x\!),\cos(\!\tfrac{2\pi}{L}x\!)\mbox{ are symmetric,}\quad
s'(\!x\!),\sin(\!\tfrac{2\pi}{L}x\!),\cos(\!\tfrac{\pi}{L}x\!)\mbox{ are skew-symmetric.}
\endeq
Take $\theta$ in the form \eq{nu2} and plug it into \eq{zerozero}; then multiply \eq{zerozero} by $\sin(\!\tfrac{\pi}{L}x\!)$ and integrate over
$(0,L)$. Several terms cancel due to \eq{skew}: if an odd number of skew-symmetric functions is involved then the integral over $(0,L)$
vanishes. We end up with the first equation in \eq{limits} with
$$\gamma_1=\frac{1}{A}\left(\frac{GK\pi^2}{2L}+\frac{2\pi^2\ell^2H_0}{L^2}\int_0^L\frac{\cos^2(\tfrac{\pi}{L}x)}{\sr^2}dx+
\frac{2AE\ell^2}{L_c}\left[\int_0^L\frac{s''(x)\sin(\tfrac{\pi}{L}x)}{\sr^3}dx\right]^2\right)>0$$
and
$$A=\ell^2\int_0^L\left(\frac{M}{3}+2m\sr\right)\sin^2\left(\frac{\pi}{L}x\right)\, dx\, .$$

Then multiply \eq{zerozero} by $\sin(\!\tfrac{2\pi}{L}x\!)$ and integrate over $(0,L)$. By using again \eq{skew}, we find the second equation in
\eq{limits} with
$$\gamma_2=\frac{1}{B}\left(\frac{2GK\pi^2}{L}+\frac{8\pi^2\ell^2H_0}{L^2}\int_0^L\frac{\cos^2(\tfrac{2\pi}{L}x)}{\sr^2}dx\right)>0\, ,\quad
B=\ell^2\int_0^L\left(\frac{M}{3}+2m\sr\right)\sin^2\left(\frac{2\pi}{L}x\right)\, dx\, .$$
The proof is complete.
\end{proof}

Needless to say, the equations \eq{limits} admit periodic (bounded!) solutions $\theta_1$ and $\theta_2$ for any initial data.
We denote by $\mathbb{N}$ the set of nonnegative integers and we assume that
\neweq{assumption0}
\frac{T_0}{\pi}\sqrt{\gamma_k}\not\in\mathbb{N}\qquad\mbox{for }k=1,2
\endeq
where $T_0=\lim_{\|y\|_\infty\to0}T_y$ is the period of the corresponding periodic solution of the linear problem.
Clearly, \eq{assumption0} has probability 1 to occur among all possible choices of $(G,K,M,T_0,L,\ell)\in\R^6_+$.\par
Let $T_y$ be as in \eq{trueperiod}; if \eq{assumption0} holds, we know that for $k=1,2$ there exists $n_k\in\mathbb{N}$ such that
\neweq{boh}
\frac{n_k^2\pi^2}{T_0^2}<\gamma_k<\frac{(n_k+1)^2\pi^2}{T_0^2}\, .
\endeq
If $y\ne0$ and $y,y_x,y_{xx}\to0$ uniformly with respect to $t$, then
$$\chi_{kk}(t)\to\gamma_k\quad(k=1,2)\, ,\qquad\chi_{12}(t)\to0\qquad\mbox{uniformly}$$
and the system \eq{twofamilies} converges to \eq{limits}, see Lemma \ref{limit}.
Therefore, by \eq{boh} and by continuity, we infer that there exists $\varepsilon>0$ such that if $\|y\|_\infty<\varepsilon$ then
$$
\frac{n_k^2\pi^2}{T_y^2}\leq \chi_{kk}(t)-|\chi_{12}(t)|\leq \chi_{kk}(t)+|\chi_{12}(t)|\leq \frac{(n_k+1)^2\pi^2}{T_y^2}
$$
for $k=1,2$ and for all $t\ge0$. Then a stability criterion for the Hill equation due to Zhukovskii \cite{zhk}, see also
\cite[Test 1 p.697, vol.2]{yakubovich}, states that the two equations in \eq{twofamilies} are strongly stable.
In turn, Lemma \ref{suff} implies that the system \eq{theta6} is strongly stable. We have so proved the following statement.

\begin{theorem}\label{conclusion}
Assume \eqref{assumption0} and let $y$ be a nonlinear longitudinal mode of \eqref{systode}. There exists $\varepsilon>0$ such that if
$\|y\|_\infty<\varepsilon$, then the projection \eqref{theta5} of \eqref{caps2} onto $\mbox{\rm span}\{\sin(\frac{\pi}{L}x),\sin(\frac{2\pi}{L}x)\}$
is strongly stable; therefore, $y$ is torsionally stable.
\end{theorem}

Clearly, a solution $y$ of \eq{systode} with small $L^\infty$-norm also has small energy. Theorem \ref{conclusion} says nothing about the stability
of \eq{theta5} when the periodic solution $y$ of \eq{nony} has large amplitude (or energy). We now address this question.\par
Instead of \eq{theta5} we consider a {\em first order} linear differential system such as
\neweq{linearmat}
\dot{Z}=A(t)Z\, ,
\endeq
where $Z\in\R^{2\nu}$ for some $\nu\ge2$ and $A$ is a $T$-periodic $2\nu\times2\nu$ matrix for some $T>0$; the system \eq{theta5} may be written in
the form \eq{linearmat} by setting $Z=(W,\dot{W})$. Define a transfer map $Z(t)\mapsto Z(t+T)$ in terms of a suitable transition matrix $M$,
that is, $Z(t+T)=M\, Z(t)$ with $Z$ being a solution of \eq{linearmat}. It is clear that if all the (complex) eigenvalues of $M$ have modulus 1,
then the solutions of \eq{linearmat} are globally bounded and therefore the trivial solution $Z\equiv0$ of \eq{linearmat} is stable. The numerical
procedure described in Section \ref{numstab} computes these eigenvalues.\par
Ortega \cite{ortega} considers a family of {\em nonlinear} Hill equations such as
\neweq{nonhill}
\ddot{z}+a(t)z+f(z)=0\qquad\mbox{with }a(t+T)=a(t)\, ;
\endeq
a typical example is $f(z)=z^3$ but many other nonlinearities $f$ yield the same behavior. Taking $Z:=(z,\dot{z})$, this equation may be
reduced to a first order nonlinear $2\times 2$ system of the kind of \eq{linearmat}:
\neweq{nonlinearmat}
\dot{Z}=A(t)Z+F(Z)\, .
\endeq
Under suitable conditions on the nonlinearity $F$, Ortega proves that if the trivial solution of \eq{linearmat} is stable,
then also the trivial solution of \eq{nonlinearmat} is stable.
This sufficient condition shows that a linearization process is legitimate, if one aims to study the stability of \eq{nonlinearmat}.\par
In their paper, Cazenave-Weissler \cite{cazw2} analyzed the stability of simple modes of \eq{nonlocw} by using techniques from the Floquet theory,
combined with Poincar\'e maps and the Hill equations. They were able to prove that, if the energy is sufficiently large, then simple modes of
\eq{nonlocw} become unstable: they take advantage of the fact that \eq{nonlocw} is ``almost linear''.\par
For a different nonlinear and local system of ODE's it was proved in \cite{bgz} that the same fact occurs, large energies yield instability of modes.
The analysis in \cite{bgz} is fairly precise and is made possible by the fact that the linearized dynamical system leads to some Mathieu
equations (a particular case of the Hill equations) for which more precise theoretical tools are available.\par
Therefore, both for nonlinear nonlocal problems \cite{cazw2} and for nonlinear local problems \cite{bgz}, the instability of the nonlinear longitudinal
modes of \eq{systode} with large energy has to be expected. The system \eq{caps1}-\eq{caps2} is certainly much more complicated than the
systems considered there, and a rigorous proof of the torsional instability at high energies seems out of reach, but we believe that it is governed by
the same phenomena. In the next section we will see that the numerical experiments support our belief.

\subsection{Thresholds of torsional instability}\label{numstab}

Take a $k$-th nonlinear longitudinal mode $y$ of \eq{systode} having period $T_y$ (see Definition \ref{nnllmm}) and use it to build the coefficient
matrix $\Xi(t)$ in \eq{theta5}. Then transform \eq{theta5} into \eq{linearmat} by setting $Z=(W,\dot{W})$. We choose $2\nu$ different initial conditions
$Z(0)=e_k$, where now $\{e_k\}_{k=1,\ldots,2\nu}$ is the canonical basis in $\mathbb{R}^{2\nu}$, and for each initial condition we solve \eq{linearmat} in
$[0,T_y]$. Let $M$ be the transition matrix obtained by placing into adjacent columns the solutions $Z(T_y)$ of \eq{linearmat} for all the initial conditions
$Z(0)=e_k$ ($k=1,...,2\nu$).
Compute its eigenvalues and consider the largest in modulus: when such value is larger than 1 the longitudinal mode is torsionally unstable.\par
In this section we implement the just described procedure and we describe the numerical computations of the thresholds of torsional instability for
each branch of periodic longitudinal modes. Several different indicators are used to measure the instability. The values of the parameters involved
in the equations are taken as for the collapsed TNB, see Section \ref{constants}.\par
Table \ref{thres} summarises the thresholds of instability for the different branches of longitudinal modes: the second column is the energy threshold,
the third column is the period of the longitudinal mode at the threshold, and the fourth is a measure of the amplitude of oscillation, that is, the
gap between the maximum and the minimum of the longitudinal mode $y$:
$$\Delta:=\max_{x,t}y(x,t)-\min_{x,t}y(x,t)\, .$$

\begin{table}[htdp]
\begin{center}
{\small
\begin{tabular}{|c|c|c|c|c|}
\hline
Branch&Energy ($MJ$)&Period ($s$)&$\Delta$ ($m$)\\
\hline
 1  & 38. & 11.22&5.8\\
 2  & 51.8 & 8.46&10\\
 3  & 15.5 & 5.48&6.5\\
 4  & 53.7 & 3.97&7.8\\
 5  & 74.1 & 3.14&6.9\\
 6 &  56.6 & 2.53&4.5\\
 7 &  91.4 & 2.18&5.2\\
 8 &  95.8 & 1.86&4.6\\
 9 &  87.1 & 1.59&3.8\\
 10 & 82.1 & 1.38&3.3\\
\hline
\end{tabular}
}
\end{center}
\caption{Thresholds of instability on each branch of nonlinear modes.}\label{thres}
\end{table}

A careful look at the video \cite{tacoma} and the data in the Report \cite{ammann} confirm that the oscillations prior to the TNB collapse were
of the order of a few meters, as in our numerical results, see the values of $\Delta$ in Table \ref{thres}. The period at the threshold of instability
should be compared with the period of the longitudinal mode on the same branch when the energy is close to 0, that is, when the nonlinear system
\eq{systode} is close to linear, see Table \ref{almostlinear}.\par
Let us now introduce an effective tool to measure the torsional instability which allows the comparison of different longitudinal modes.
The Floquet Theorem (see \cite[Theorem 8.23]{chicone}) states that, if $\Psi(t)$ is a fundamental matrix solution of \eq{linearmat}, then there exists
a matrix $B$ and a $T_y$\,-\,periodic matrix $P$ such that $\Psi(t)=P(t)e^{tB}$; note that both $B$ and $P$ may be complex-valued. The eigenvalues
$\lambda_1,...,\lambda_{2\nu}$ of $e^{T_yB}$ are called the {\em characteristic multipliers} of \eq{linearmat}: if $V_1,...,V_{2\nu}$ denote the
corresponding normalized eigenvectors, then the solution $Z_j$ of \eq{linearmat} satisfying the initial condition $Z_j(0)=V_j$ (for some $j=1,...,2\nu$)
also satisfies $Z_j(T_y)=\lambda_jZ_j(0)$ and, in turn, $Z_j(mT_y)=\lambda_j^mZ_j(0)$ for any integer $m\ge1$. Whence $|\lambda_j|^{1/T_y}$ represents
the rate of growth of the amplitude of oscillation of $Z_j$. This leads to the following definition:

\begin{definition}
We call {\bf expansion rate} of \eqref{linearmat} the largest rate of growth for solutions of \eqref{linearmat}:
$$\mathcal{ER}:=\max_j\ |\lambda_j|^{1/T_y}\, .$$
\end{definition}

For every branch of nonlinear longitudinal modes we compute the expansion rate of \eq{linearmat}, that is, the expansion rate of the torsional
component of the motion. Figure \ref{evs} displays the
\begin{figure}[htdp]
\begin{center}
{\includegraphics[height=35mm, width=75mm]{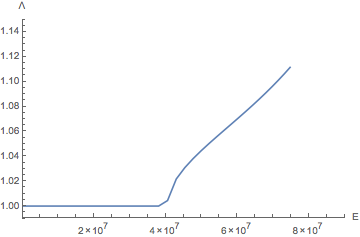}}
{\includegraphics[height=35mm, width=75mm]{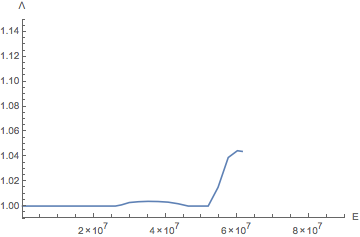}}
{\includegraphics[height=35mm, width=75mm]{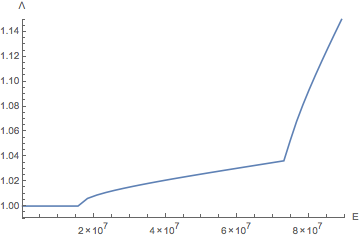}}
{\includegraphics[height=35mm, width=75mm]{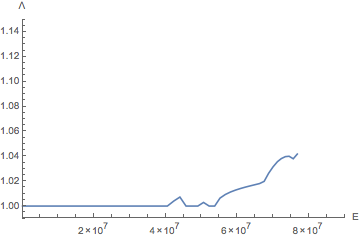}}
\caption{Expansion rate for modes 1 to 4 with respect to the energy.}\label{evs}
\end{center}
\end{figure}
expansion rate as a function of the energy for the branches of nonlinear longitudinal modes
of order 1 to 4 whereas Figure \ref{evs2} displays the same for the modes of order 5 to 10.
\begin{figure}[htdp]
\begin{center}
{\includegraphics[height=35mm, width=75mm]{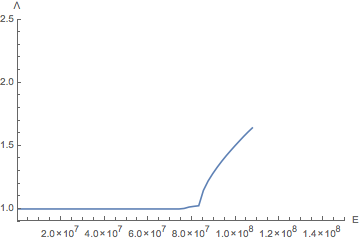}}
{\includegraphics[height=35mm, width=75mm]{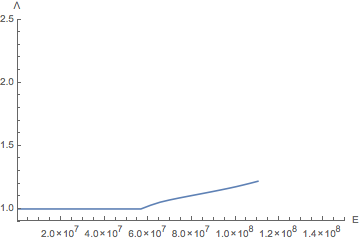}}\\
{\includegraphics[height=35mm, width=75mm]{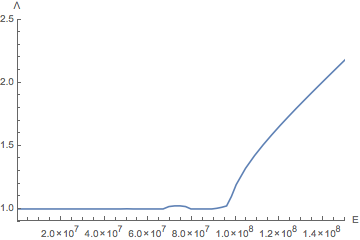}}
{\includegraphics[height=35mm, width=75mm]{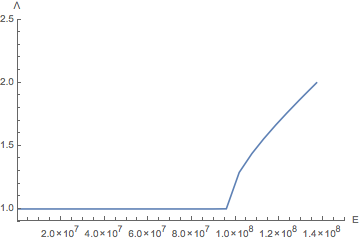}}\\
{\includegraphics[height=35mm, width=75mm]{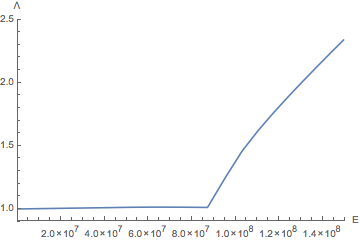}}
{\includegraphics[height=35mm, width=75mm]{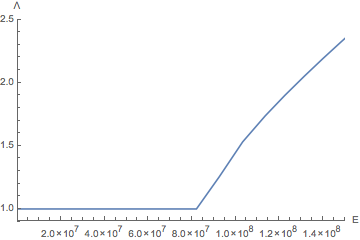}}
\caption{Expansion rate for modes 5 to 10 with respect to the energy.}\label{evs2}
\end{center}
\end{figure}
Note that the vertical axis (the $\mathcal{ER}$-axis) has different scales in the two figures. The reason is that $\mathcal{ER}$ is much larger for
high nonlinear longitudinal modes. All the plots clearly show that there exists a threshold above which the torsional stability of the longitudinal
mode is lost, that is, where the expansion rate exceeds $1$.\par
Table \ref{ER} summarises the expansion rates on all the branches that we considered, for different values of the energy. Some values in Table \ref{ER}
are not available because the Newton algorithm would require too small increments of $T$ in order to follow the branch of solutions. Note that the energy
of a longitudinal mode is related to the amplitude of oscillation $\Delta$ and, for a given energy, higher modes have smaller amplitudes.
Therefore, for the same amplitude of oscillation, the expansion rate of the 9th and 10th longitudinal mode is larger than that
of the 5th, see also Table \ref{thres}.

\begin{table}[htdp]
\begin{center}
{\small
\begin{tabular}{|c|c|c|c|c|c|c|c|}
\hline
Branch&2&4&6&8&10&12&14\\
\hline
 1 & 1. & 1. & 1.0662 & --- & --- & --- & --- \\
 2 & 1. & 1.00365 & 1.03904 & --- & --- & --- & --- \\
 3 & 1.00614 & 1.02071 & 1.02961 & 1.08141 & 1.20949 & --- & --- \\
 4 & 1. & 1. & 1.01287 & --- & --- & --- & --- \\
 5 & 1. & 1. & 1.00001 & 1.01521 & 1.50051 & --- & --- \\
 6 & 1. & 1. & 1. & 1.09919 & 1.16332 & --- & --- \\
 7 & 1. & 1. & 1. & 1. & 1.09852 & 1.58567 & 1.97158 \\
 8 & 1. & 1. & 1. & 1. & 1.00112 & 1.66552 & --- \\
 9 & 1. & 1. & 1.01322 & 1.01353 & 1.24852 & 1.76429 & 2.12488 \\
 10 & 1. & 1. & 1. & 1. & 1.25447 & 1.73715 & 2.05263 \\
\hline
\end{tabular}
}
\end{center}
\caption{Expansion rates versus energy (in MJ).}\label{ER}
\end{table}\vfill\eject

According to Eldridge \cite[V-3]{ammann}, a witness on the day of the TNB collapse, {\em the bridge appeared to be behaving in the customary
manner} and the motions {\em were considerably less than had occurred many times before}. From \cite[p.20]{ammann} we also learn that in the months
prior to the collapse {\em one principal mode of oscillation prevailed} and that {\em the modes of oscillation frequently changed}.
On the day of the collapse, the torsional oscillations at the TNB started suddenly and, according to Farquharson \cite[V-10]{ammann}, {\em the motions,
which a moment before had involved a number of waves (nine or ten) had shifted almost instantly to two}. These observations show that the transfer of
energy (from longitudinal to torsional) depends on the particular excited mode.\par
Table \ref{ER} explains why torsional oscillations did not appear earlier at the TNB even in presence of larger longitudinal oscillations,
as reported by Eldridge: there are longitudinal modes which have a very small expansion rate. Moreover, Table \ref{ER} also explains why torsional
oscillations appeared with smaller longitudinal oscillations: the 9th and 10th longitudinal mode appear more prone to generate torsional oscillations
because they have large expansion rate $\mathcal{ER}$.

\subsection{Glossary and  constants}\label{constants}

In this section we list the constants considered throughout the paper.

\begin{itemize}
\item $y=y(x,t)$ downwards vertical displacement of the center of the deck
\item $\theta=\theta(x,t)$ torsional angle of the deck
\item $\ell$ half width of the deck ($6m$)
\item $L$ length of the deck ($853.44m$)
\item $M$ mass linear density of the deck  ($7198kg/m$)
\item $E$ Young modulus of the deck and of the cables ($210GPa$)
\item $I$ linear density of the moment of inertia of the cross section
($0.15m^4$)
\item $G$ shear modulus of the deck ($81GPa$)
\item $K$ torsional constant of the deck ($6.44\cdot10^{-6}m^4$)
\item $m$ mass linear density of the cable ($981kg/m$)
\item $H_0$ horizontal component of the tension of the cables ($58300kN$)
\item $H(x)$ tangential component of the tension of the cables
\item $L_c$ length of the cable ($868.62m$)
\item $s_0$ length of the longest hanger ($72m$)
\item $A$ area of the section of the cable ($0.1228m^2$)
\item Bridge \emph{at rest} means bridge fully mounted, operational at equilibrium.
\item Bridge \emph{unloaded} means that the deck is not mounted.
\end{itemize}

\section{Conclusions}\label{conclu}

For long time the aerodynamic effects have been considered the only cause of the TNB collapse. But, as quickly surveyed in the Introduction, no
purely aerodynamic explanation is nowadays able to justify the origin of the torsional oscillation, which is the main culprit for the collapse.
In this paper we suggest a new nonlinear model, which improves \cite{agmjm}, and we focus our attention on the nonlinear structural behavior
of suspension bridges. To this end, we follow a suggestion by Irvine \cite{irvine} and we isolate the bridge from aerodynamic effects and dissipation.
The new model and the resulting evolution PDEs which describe the dynamics of suspension bridges are presented in Section \ref{cabbeam}.\par
In Section \ref{nlm} we explain how to find periodic longitudinal oscillations as solutions of a nonlinear system of ODEs: contrary to linear systems,
the period depends on the energy and on the amplitude of the oscillation.\par
In Section \ref{theores} we describe the theoretical framework governing the structural instability of suspension bridges. Our analysis is based
on tools from the Floquet theory and a sufficient condition for the stability at low energy is obtained in Section \ref{nu22}.
Then, in Section \ref{numstab}, we put into the model the quantitative parameters of the collapsed TNB and we compute its thresholds of instability.
The results are in line with the values observed on the day of the collapse and well fit the Report \cite{ammann}. First, the computed critical
amplitude of longitudinal oscillations is the same as in the video \cite{tacoma}. Second, our results explain why the TNB withstood larger longitudinal
oscillations on low modes, but failed for smaller longitudinal oscillations on higher modes,
see \cite[pp.28-31]{ammann}. Our results show that the origin of torsional instability is structural.\par
It is clear that, in absence of wind or external loads, the deck of a bridge remains still. On the other hand, when the wind hits a
bluff body (such as the deck of the TNB) the flow is modified and goes around the body. Behind the deck, or a ``hidden part'' of it,
the flow creates vortices which are, in general, asymmetric. This asymmetry generates a forcing lift which starts the vertical oscillations of the deck.
Up to some minor details, this explanation is shared by the whole community and it has been studied with great precision in wind tunnel tests,
see e.g.\ \cite{larsen,scanlan,scott}. The vortices induced by the wind increase the internal energy of the structure and generate wide longitudinal
oscillations which look periodic in time. This is the point where our analysis starts, that is when the longitudinal oscillations of the structure reach a
periodic motion which is maintained in amplitude by a somehow perfect equilibrium between the input of energy from the wind and internal dissipation.
Throughout this paper we have seen that, in such situation, if the longitudinal oscillation is sufficiently large, then a structural instability appears:
this is the onset of torsional oscillations. In order to obtain a more accurate description of the dynamics of suspension bridges, the next step should
be to take into account also the behavior of the aerodynamic and dissipation effects after the appearance of the torsional oscillations.\par\medskip\noindent
\textbf{Acknowledgements.} The Authors are partially supported by the PRIN project {\em Equazioni alle derivate parziali di tipo ellittico
e parabolico: aspetti geometrici, disuguaglianze collegate, e applicazioni}, and they are members
of the Gruppo Nazionale per l'Analisi Matematica, la Probabilit\`a e le loro Applicazioni (GNAMPA)
of the Istituto Nazionale di Alta Matematica (INdAM).

\end{document}